\documentclass[12pt]{amsart}
\usepackage{latexsym,amsmath,amssymb,esint,enumitem,bbm}
\usepackage{mathtools,slashed}
\usepackage{graphicx,subfigure,psfrag}
\usepackage{hyperref}
\usepackage{xcolor}

\newcommand{\RR}{{\mathbb R}}
\newcommand{\One}{{\mathbbm 1}}
\DeclareMathOperator{\inte}{int}
\DeclareMathOperator{\supp}{supp}

\usepackage{times}

\newtheorem{theorem}{Theorem}[section]
\newtheorem{Theorem}{Theorem}
\newtheorem*{theorem*}{Theorem}
\newtheorem{lemma}[theorem]{Lemma}
\newtheorem{corr}[theorem]{Corollary}

\theoremstyle{definition}

\newtheorem*{deff*}{Definition}
\newtheorem{remark}[theorem]{Remark}

\numberwithin{equation}{section}

\newcommand{\abs}[1]{\left\lvert #1 \right\rvert}
\definecolor{airforceblue}{rgb}{0.36, 0.54, 0.66}
\definecolor{amber}{rgb}{1.0, 0.49, 0.0}
\definecolor{crimson}{rgb}{0.86, 0.08, 0.24}
\definecolor{amethyst}{rgb}{0.6, 0.4, 0.8}

\begin{document}
\title[Strict concavity properties of cross covariograms] {Strict concavity properties of cross covariograms}
\author {Gabriele Bianchi}
\address{(G.B.) Universit\`a di Firenze,
Dipartimento di Matematica e Informatica
`Ulisse Dini', Viale Morgagni 67/a, 50139, Firenze, Italy.}
\email{\tt gabriele.bianchi@unifi.it}

\author{Almut Burchard}
\address{(A.B.) University of Toronto,
Department of Mathematics, 40 St. George Street,
{Rm. 6190},
Toronto, ON M5S 2E4, Canada.}
\email{\tt almut@math.toronto.edu}

\author{Lawrence Lin}
\address{(L.L.) University of Toronto,
Department of Mathematics, 40 St. George Street,
{Rm. 6190},
Toronto, ON M5S 2E4, Canada.}
\email{\tt lawrence.lin049@gmail.com}

\begin{abstract}
It is well-known that
the cross covariogram of two convex bodies in $\RR^n$
is $1/n$-concave on its support.
This paper provides conditions for strict $1/n$-concavity
in dimension $n>1$, and an analysis
of how it can fail.
Among the implications are
that ({\em i.})
the cross covariogram of strictly
convex bodies is strictly $1/n$-concave, unless
one body contains a translate of the other in its
interior; and ({\em ii.}) the cross covariogram of an
arbitrary convex body with its reflection through the origin
has strictly convex level sets..
\end{abstract}

\maketitle

\thispagestyle{empty}

\section{Introduction and main results}
\label{sec:intro}

{\em What is the optimal position for two sets
that maximizes the overlap
between them?}
This is a basic geometric problem that
appears in many applications.
We consider the problem of finding an optimal translation;
the same question could be asked about other transformations,
such as 
rotations and  dilations. Typically,
optimal positions exist, but can be hard to identify.

We focus on the special case of convex bodies,
that is, compact convex subsets of $\RR^n$
with non-empty interior. In this setting,
the problem is to maximize the cross covariogram
\[
g_{K,L}(x):= \left|K\cap (L+x)\right|\,,\quad  x\in\RR^n
\]
over the translation vector, $x$.
We came across the question
of unique maximizers of $g_{K,L}$
when considering equality cases for
convolution-type inequalities
that were recently proved by Haddad~\cite{H23}.
Here we address the broader question:
{\em When is a suitable power of $g_{K,L}$ strictly concave?}

The covariogram was introduced in the
1970s by Matheron as a geometric counterpart to
the variance of random variables~\cite{M75}.
It has found applications in
the statistical analysis of geometrical data
ranging from crystallography and microscopy to physical geography.
A question that has received some attention
in the mathematics literature is Matheron's covariogram problem:
{\em How much information about a convex body $K$ is encoded
in its covariogram $g_{K,K}$?}
This problem has largely been resolved in the last
20 years. A planar convex body is determined
by its covariogram
uniquely up to translation and reflection~\cite{AB09};
the same is true for convex polytopes in dimension $n=3$,
but not in higher dimensions.
Recent developments also include algorithms
for reconstructing the body~\cite{BGK11},
and a characterization of the functions that can be realized
as covariograms~\cite{GL16}. For more information
about covariograms and related quantities
we refer the reader to~\cite{Bianchi2023}
and the references therein.

The cross covariogram
of two convex bodies $K, L$ is
continuous, supported on the sumset $K+ (-L)$,
and attains its maximum on some non-empty compact
convex set.  As the convolution of the
indicator functions $\One_{K}$ and $\One_{-L}$,
which are logarithmically concave, it is log-concave by the
Pr\'ekopa-Leindler inequality~\cite{L72,P73}.
In particular, it has convex level sets
and decreases monotonically along rays that emanate
from points where the maximum is attained.

The Pr\'ekopa-Leindler inequality
is a distinguished member of the family of Borell-Brascamp-Lieb
inequalities, which extend the Brunn-Minkowski inequality
from sets to nonnegative functions
on $\RR^n$~\cite{Bo75,BL76,McC94}
(see also~\cite{CMS01,Mel19,AAFS20,PiR21} for more recent improvements).
The family is indexed by a parameter
$\alpha\in [-1/n,\infty]$ that
signifies the exponent in
a generalized mean.
The inequalities bound integrals over
the $\alpha$-means of two functions 
by the $\gamma$-means
of their integrals, 
where $\gamma=\alpha/(1+n\alpha)$.
The endpoint case with $\alpha=-1/n$ implies
the general inequality by rescaling and taking limits.
The Pr\'ekopa-Leindler theorem
is the case $\alpha=0$, where
the inequality is dimension-free and involves the geometric mean.

In this paper we study conditions for
{\em strict} $1/n$-concavity (and log-concavity)
of the cross covariogram.  The exponent $1/n$
corresponds to the endpoint case
of the Borell-Brascamp-Lieb inequality
with $\alpha=+\infty$, valid on the
smallest concavity class which is comprised
of positive multiples of indicator functions
of convex sets (see~\cite[Theorem 3.3]{BL76}).
Our first result provides a convenient sufficient condition.

\begin{Theorem} [Strict $1/n$-concavity]
\label{thm:global}
Let $K,L\subset\RR^n$
be convex bodies, where \mbox{$n>1$},
and let $g_{K,L}$ be their cross covariogram.
If $K,L$ are strictly convex and
\begin{equation}
\label{eq:max}
\max g_{K,L}<\min\{|K|,|L|\},
\end{equation}
then $g_{K,L}^{1/n}$ is strictly concave on its support.
\end{Theorem}

Thus the convolution $g_{K,L}=\One_K*\One_{-L}$
is strictly $1/n$-concave, even though
the indicator functions $\One_K$ and $\One_{-L}$ are
mostly flat. It follows that
$g_{K,L}$ has a unique maximum,
and its level sets $\{x\in\RR^n\mid g_{K,L}(x)\ge h\}$ are strictly convex
for $0<h<\max g_{K,L}$.

In dimension $n>1$,
Eq.~\eqref{eq:max} ensures that the boundaries
$\partial K$ and $\partial (L+x)$
intersect non-trivially whenever $g_{K,L}(x)>0$.
For symmetric convex bodies,
this assumption can be replaced by the sharper condition
that $\partial K \cap \partial L$
is non-empty (Corollary~\ref{corr:symmetric}).

Strict convexity of the bodies is not always necessary,
as evidenced by the following result
of Meyer, Reisner, and Schmuckenschl\"ager~\cite[Prop. 2.2]{MRS93}:

\begin{theorem*} [Covariograms of symmetric bodies~\cite{MRS93}]
Let $K\subset\RR^n$ be a convex body
such that $-K=K$, and let $g_{K,K}$ be its covariogram.
Then the level sets
\[
\{x\in\RR^n\mid g_{K,K}(x)\ge h\},\qquad (0<h<|K|)
\]
are strictly convex.
\end{theorem*}

\noindent {Note that $g_{K,K}$ attains its
maximum $|K|$ only at $x=0$.}

\smallskip
Our main results concern failures of strict
$1/n$-concavity of cross covariograms.
In principle, these can be understood
in terms of equality cases in the
Borell-Brascamp-Lieb inequality,
applied to the convolution
$\One_K*\One_{-L}$.  The literature contains  a number of
results on equality and near-equality cases of the
Borell-Brascamp-Lieb inequalities,
including Dubuc's classical result~\cite{D77} (for $\alpha=-1/n$),
\cite{BallBo10,BuF14,BorFR23} (for $\alpha=0$),
\cite{Rossi2017} (for a full range exponents), \cite{BK2018}
(on curved spaces), and many others.
To summarize, equality in a Borell-Brascamp-Lieb
inequality for a pair of functions
generally implies that the functions are related by composition
with a homothety. However,
it is not obvious what that means
for strict concavity properties of convolutions, especially
convolutions of functions that vanish in some areas.
The issue is that the $\alpha$-mean of two numbers
vanishes if just one of them is zero.
This makes it difficult to untangle
how each factor contributes to the concavity properties
of the convolution.

We investigate, for the special case of
the cross covariogram
$g_{K,L}=\One_K*\One_{-L}$, how strict
$1/n$-concavity of the convolution
is determined by {\em geometric} properties of the factors.
To motivate our results, consider a pair of (possibly unbounded)
closed convex sets $K, L\subset\RR^n$
such that $g_{K,L}(x)<\infty$ for all $x\in\RR^n$
and $g_{K,L}$ does not vanish identically.
There are three obvious scenarios
where $g^{1/n}_{K,L}$ is affine on some subset:

\begin{figure}[t]
    \centering
    \includegraphics[height=2.5cm]{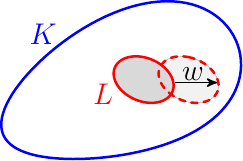}
\hspace{1cm}
    \includegraphics[height=2.5cm]{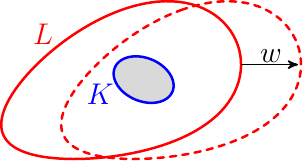}
    \caption{\small
Examples where $g_{K,L}$ has a plateau
at height $|L|$ or $|K|$
}
\label{fig:containment}
\end{figure}

\smallskip

\begin{enumerate}
\item {\bf Compact containment.} \
If $K$ contains $L$ in its interior
and $L$ has positive volume,
then the graph of $g_{K,L}$ has a plateau at 
{height $|L|$}
(see Fig.~\ref{fig:containment}, left).

\smallskip

\item {\bf Cylinders.} \
If $K$ is invariant under translation in some direction,
then $g_{K,L}$ is constant on lines in that direction
(see Fig.~\ref{fig:cylinder}, left).

\smallskip

\item {\bf  Pairs of cones.}\
If $K$, $L$ are closed convex cones in $\RR^n$
that are disjoint except for a
common vertex, then $g_{K,L}^{1/n}$ is
affine along rays through that vertex
(see Fig.~\ref{fig:cones}).
\end{enumerate}

\smallskip
Scenarios (2) and (3) include as
a special case the situation where $K$
is a half-space and $L$ a convex cone,
which are disjoint except that
the boundary of $K$ contains the
vertex of $L$.  Translations parallel to the boundary of the
half-space fall under Scenario (2), and
transversal translations under Scenario~(3).
The unbounded cylinders and cones can be
truncated to produce examples of
convex bodies whose cross covariogram
fails to be strictly $1/n$-concave on certain line segments.

We find that these are essentially
the only obstructions to strict $1/n$-concavity
(up to translation and interchange of $K$ and $L$).
For instance, in Theorem 1, the condition
that $\max g_{K,L}<\min\{|K|,|L|\}$ rules out Scenario~(1),
and the strict convexity of the bodies rules out
Scenarios (2-3).
We note in passing that since the cross covariogram is
$1/n$-concave on its support, it is strictly
log-concave there (in fact strictly $\alpha$-concave
for every $\alpha<1/n$) unless it is constant on
some line segment of positive length.
{Thus $\One_K*\One_{-L}$ is
strictly log-concave on its support in Scenario (3), but not in
Scenarios (1-2).}

Our next result addresses Scenario (2).

\begin{figure}[t]
    \centering
    \includegraphics[height=2.4cm]{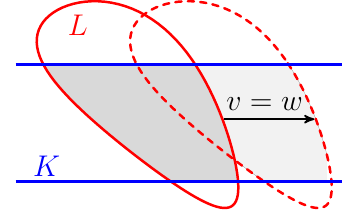}\qquad
    \includegraphics[height=2.4cm]{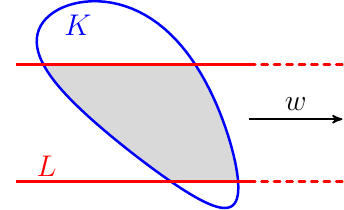}\qquad
    \includegraphics[height=2.5cm]{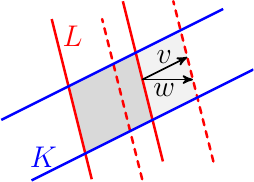}
    \caption{\small Examples where $g_{K,L}$ is constant along a line
segment in the direction of $w$. The first two images
correspond to the cases $v=w$ and $v=0$ of
Theorem~\ref{thm:constant}; in the last image $v\not\in \{0,w\}$
}
\label{fig:cylinder}
\end{figure}

\begin{Theorem} [$g_{K,L}$ constant on a line segment]
\label{thm:constant}
Let $K,L\subset\RR^n$ be closed convex sets, where $n>1$,
and let $a, w\in\RR^n$.
Assume that $0<g(a)<\infty$ and $w\ne 0$. If
\[
g_{k,L}(a\pm w) = g_{K,L}(a),
\]
then there exists a vector $v\in\RR^n$ such that
\begin{equation}
\label{eq:KL-prime}
K\cap (L+a+tw)= K'\cap (L'+tw),  \qquad (-1\le t\le 1),
\end{equation}
where
\begin{align*}
K'&:=\left\{x+tv \mid x\in K, t\in\RR\right\} \\
L'&:=\left \{x+ t(v\!-\!w)\mid x\in L+a, t\in\RR\right\}.
\end{align*}
\end{Theorem}
Since $K'$ and $L'$ are invariant under translation by $v$
and $w-v$, respectively,
Eq.~\eqref{eq:KL-prime} says that, for $-1\le t\le 1$,
\begin{align}
\label{eq:KL-prime-b}
\begin{split}
K\cap (L+a+tw) &= (K'\cap L')+t v\\
&=(K\cap (L+a))+t v.
\end{split}
\end{align}
Here the second equality comes from the first one
and Eq.~\eqref{eq:KL-prime} with $t=0$.
Formula~\eqref{eq:KL-prime-b} shows that
for each $t\in[-1,1]$ the intersection
$K\cap (L+a+tw)$ is the translation by $tv$
of a  body that does not depend on $t$.

If $v\ne 0$ then $K'$ is the convex cylinder
in the direction of $v$ generated by $K$.
In this case, Eq.~\eqref{eq:KL-prime} implies
that with every point $z\in \partial K \cap (L+a)$,
the boundary of $K$ contains the entire
line segment with endpoints $z\pm v$.
This is a strong constraint unless
$\partial K \cap (L+a)$ is empty, in which
case $K$ contains $L+a$ in its interior
(see Lemma~\ref{lem:dKdL} below).
If $v=0$ then $K'=K$ and Eq.~\eqref{eq:KL-prime}
places no condition on the shape of $K$ (see Fig.~\ref{fig:cylinder}, middle).

Correspondingly, if $v\ne w$
then for every point $z\in K \cap \partial (L+a)$
the boundary of $L$ contains the
line segment with endpoints $z-a\pm v$.
If $v=w$ there is no condition on the shape of~$L$.
{Since $w\ne 0$ by assumption,}
$v$ and $v-w$ cannot vanish simultaneously.
If both are non-zero, then $g_{K,L}$ is constant on
a parallelogram with sides
parallel to $v$ and $w-v$ (see
Fig.~\ref{fig:cylinder}, right
and Remark~\ref{remark:st}).

\begin{figure}[t]
    \centering
    \includegraphics[height=3cm]{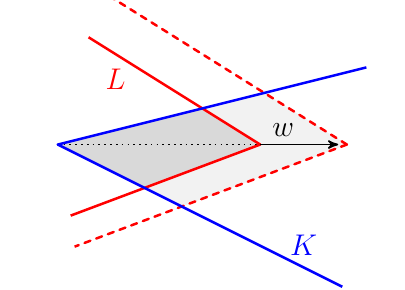}\quad
    \includegraphics[height=3cm]{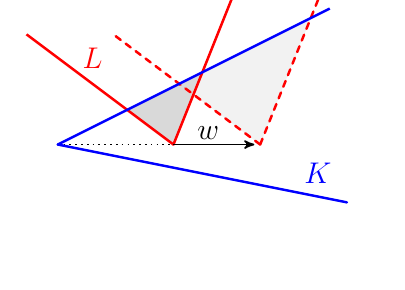}\quad
    \includegraphics[height=3cm]{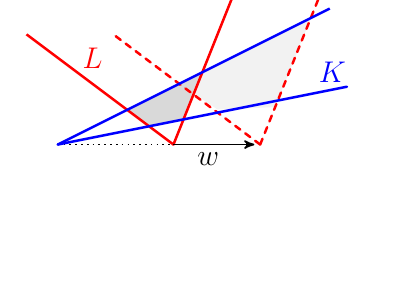}
    \caption{\small Examples where $g_{K,L}^{1/n}$ is
affine and strictly increasing along a line segment
with endpoints $a\pm w$.  The vector $w$
points from the vertex of $K$ towards the vertex of $L+a$}
\label{fig:cones}
\end{figure}

We turn to Scenario (3).

\begin{Theorem} [$g_{K,L}^{1/n}$ non-constant affine
on a line segment] \label{thm:affine}
Let $K,L\subset\RR^n$ be
closed convex sets, where $n>1$,
and $a,w\in\RR^n$.
Assume that $0<g_{K,L}(a)<\infty$ and $w\ne 0$.  If
\[
g_{K,L}^{1/n}(a\pm w) =
(1\pm \lambda)\, g_{K,L}^{1/n}(a),
\]
for some  $\lambda\in (0,1]$,
then there exists $v\in\RR^n$ such that
\begin{equation}
\label{eq:KL-pprime}
K\cap (L+a+tw)= K''\cap (L''+ tw)
\qquad (-1\le t\le 1),
\end{equation}
where
\begin{align*}
K''&:=\left\{-\tfrac1\lambda v +t(x+\tfrac1\lambda v)
\,\big\vert\, x\in K, t\ge 0 \right\}\\
L''&:=\left\{-\tfrac1\lambda(v\!-\!w)+
t(x+\tfrac1\lambda(v\!-\!w))\,\big\vert\, x\in L+a, t\ge 0 \right\}.
\end{align*}
\end{Theorem}

The sets $K''$ and $L''$ are the convex cones with
vertices $-\frac1\lambda v$ and $-\frac1\lambda(v\!-\!w)$
generated by $K$ and $L+a$, respectively.
Since the cones are invariant under dilation from
their vertices,  Eq.~\eqref{eq:KL-pprime} implies that,
for $-1\le t\le 1$,
\begin{align}\label{eq:KL-pprime-b}
\begin{split}
K\cap (L+a+tw) &= (1+t \lambda)(K''\cap L'') + tv\\
&=(1+t \lambda)(K\cap (L+a)) + tv.
\end{split}
\end{align}
Here the second equality comes from the first one
and Eq.~\eqref{eq:KL-pprime} with $t=0$.
Formula~\eqref{eq:KL-pprime-b} shows that for
$t\in[-1,1]$ the intersections $K\cap (L+a+tw)$
are homothetic to a body that does not depend on $t$.

Furthermore, with every point $z\in \partial K\cap (L+a)$
different from the vertex of $K''$,
the boundary of $K$ contains a segment of the ray
from $z$ to the vertex of $K''$.
The corresponding statement holds
for points in $K\cap\partial(L+a)$.

In combination, Theorems~\ref{thm:constant} and~\ref{thm:affine}
refine Theorem~\ref{thm:global} by providing a localized
geometric view of situations where strict $1/n$-concavity fails.
We can recover Theorem~\ref{thm:global} as follows:
Under the assumptions of Theorem~\ref{thm:global},
$\partial K\cap (L+a)$ and $K\cap\partial(L+a)$
are non-empty whenever $g(a)>0$ (Lemma~\ref{lem:dKdL}).
By strict convexity, the boundaries
of $K$ and $L$ contain no line segments of positive length.
This means that neither Eq.~\eqref{eq:KL-prime}
nor Eq.~\eqref{eq:KL-pprime} can hold
for any $v, w$ with
$w\ne 0$. Since the conclusions of
Theorems~\ref{thm:constant} and~\ref{thm:affine}
are not satisfied, $g^{1/n}$ is strictly
concave on every line segment of positive length
in its support.

\bigskip \noindent
{\bf Outline of the proofs.}
All three theorems rely on the characterization
of the equality cases in the Brunn-Minkowski inequality
(Section~\ref{sec:prelims}).
For Theorem~\ref{thm:global},
we exploit that by strict convexity, every element of
$\partial K\cap \partial(L+x)$
is an extreme point of both $K$ and $L+x$ (Section~\ref{sec:global}).
Similar arguments
(specialized to the case where $K$ is symmetric and $L=K$)
were used by Meyer, Reisner, and Schmuckenschl\"ager
to prove their result described above.

{For the proof of
Theorem~\ref{thm:constant}, we show that
$g_{K,L}$ can be constant on a line segment
only if the boundary of either
$K$  (if $v\ne 0$) or $L$ (if $v\ne w$)
contains a family of parallel line segments.
This family generates the cylinder $K'$ or $L'$
(Section~\ref{sec:cylinder}).}

The proof of Theorem~\ref{thm:affine} proceeds
in the same way.  {If $g_{K,L}^{1/n}$ is
non-constant affine on a line segment,
then $K$ and $L$ each contain a family
of line segments that generate the cones $K''$ and $L''$
(Section~\ref{sec:cones})}.
The conclusions of Theorems~\ref{thm:constant} and~\ref{thm:affine}
yield useful sufficient conditions for strict
$1/n$-concavity (Section~\ref{sec:examples}).


\section{Preliminaries and notation}
\label{sec:prelims}

To fix terms, we recall some standard conventions.
A {\bf convex body} is a compact
convex set $A\subset\RR^n$ with non-empty interior.
The body is {\bf strictly} convex, if
for any two distinct points $x, y$ in $A$ and
every $t\in (0,1)$, the convex combination
$(1-t)x+ty$ lies in the interior of $A$.
We use the shorthand $tA$ for the
dilation of $A$ by a factor $t>0$, and $-A$
for its reflection through the origin.
Similarly, $A+x$ denotes the translation of
$A$ by a vector $x\in\RR^n$.
A convex body $B$ is {\bf homothetic} to $A$
if $B=tA+x$ for a suitable  $t>0$ and $x\in \RR^n$.
Finally, $|A|$ denotes
the volume of $A$,
$\inte A$ its interior,
and $\partial A$ its boundary. Moreover
$\supp f$ denotes the {support} of a function $f$.

The {\bf Brunn-Minkowski inequality}
states that for any pair of non-empty convex sets $A,B\subset\RR^n$
\begin{equation}
\label{eq:BM}
|(1-t)A + tB|^{1/n} \ge (1-t)|A|^{1/n} + t |B|^{1/n}
\qquad (0\le t\le 1),
\end{equation}
in particular, if the left hand side is
finite, then the right hand side is finite as well.
Under the additional assumption
that $A,B$ are compact and that at least one of them has positive measure,
Eq.~\eqref{eq:BM} holds with equality
for some $t\in (0,1)$, if and only if
either $A$ and $B$ are homothetic or
one of them is a single point.  In either case, all
of the sets $(1-t)A+tB$ for $0<t<1$ are homothetic.

We will frequently consider
{the midpoint at} $t=1/2$:
For $C:=\frac12(A+B)$, the Brunn-Minkowski inequality
takes the form
\[
|C|^{1/n} \ge \frac12\left( |A|^{1/n} +|B|^{1/n}\right)\,,
\]
with equality if and only if there exist
$v\in\RR^n$ and $\lambda\in [-1,1]$ such that
\begin{equation} \label{eq:BM-equal}
A = (1-\lambda)C-v\,, \qquad
B = (1+\lambda)C+v\,.
\end{equation}
The value of $\lambda$ is determined by the volume
ratio $|A|:|B|$, and the vector~$v$ by the distance vector
between
the barycenters of $A$ and $B$.

The {\bf cross covariogram} of
two convex sets $K$, $L$, is defined by
\begin{equation}
\label{def-gKL}
g_{K,L}(x) :=|K\cap (L+x)|,\qquad x\in\RR^n\,.
\end{equation}
It is always well-defined, but may
be $+\infty$ if $K$ and $L$ are unbounded.
{When there is no danger of confusion,
we drop the subscripts
and simply write $g$ for the cross covariogram
of $K$ and $L$.}

\bigskip\noindent {\bf Standing Hypothesis.} \
To avoid trivialities, we generally
assume that

\smallskip
\begin{itemize}
\item $n>1$,
\smallskip
\item $K,L\subset\RR^n$ are closed convex sets, {and}
\smallskip
\item there is a point $a\in\RR^n$ such that
$0<g_{K,L}(a)<\infty$.
\end{itemize}

\smallskip
These assumptions allow
for arbitrary pairs of convex bodies,
as well as for certain pairs of unbounded
sets such as the cylinders and cones
that were described in the introduction.
The last condition demands that $K\cap (L+a)$ is a convex body
for {\em some} choice of $a\in\RR^n$.
The next lemma proves that then
$g_{K,L}(x)$ is finite for every $x\in\RR^n$.

\begin{lemma}\label{lem:g-support}
Let $g$ be the cross covariogram
of a pair of convex sets $K,L\subset\RR^n$
that satisfy the standing hypothesis.
Then
$$
x\in \supp g_{K,L} \ \Longleftrightarrow \ K\cap (L+x)\ne \emptyset\,.
$$
Moreover, $g(x)<\infty$ for all $x\in\RR^n$.
\end{lemma}

\begin{proof}
($\Rightarrow$) \
Consider a point $x\in\supp g$ and
let $\{x_k\}$ be a sequence
convergent to $x$ such that $g(x_k)>0$, for each $n$.
The set $K\cap(L+x_k)$ is nonempty: let $w_k$ be a point
in this set. We claim that the sequence $\{w_k\}$ is bounded.

Assume the contrary. Up to extracting
a subsequence, $\|w_j\|\to\infty$ and there is $u\in S^{n-1}$
such $w_k/\|w_k\|\to u$. The convexity of $K$ implies that,
for any $z\in K$, the ray $z+\{tu, t\geq 0\}$
is contained in $K$. The sequence $\{w_k-x_k+a\}$ is contained
in $L+a$, $\|w_k-x_k+a\|\to\infty$ and $(w_k-x_k+a)/\|w_k-x_k+a\|\to u$.
Thus if $B$ is any ball contained in $K\cap(L+a)$, the set
$B+\{tu, t\geq 0\}$ is also contained in $K\cap(L+a)$.
This contradicts the hypothesis $g(a)<\infty$ and
proves that $\{w_k\}$ is bounded.
After extracting a further subsequence,
we may assume that $\{w_k\}$ has a limit. This limit belongs
to $K\cap(L+x)$, proving that this set is nonempty.

($\Leftarrow$) \ Conversely, suppose that
$K\cap (L+x)$ contains a point, $w$. Since
$K$ and $L$ are convex sets with non-empty interior,
there exist sequences
$\{y_k\}\subset \inte K$ and $\{z_k|\subset \inte (L+x)$ such that
$\lim y_k=\lim z_k=w$. Therefore
$$ g(x+y_k-z_k)= |(K-x_k)\cap (L+x-z_k)|>0,$$
showing that $x\in\supp g$.

It remains to argue that $g(x)$ is always finite.
The first part of the proof
(applied to the constant sequence with $x_k=x$, $k\ge 1$)
yields that every sequence in $K\cap (L+x)$
is bounded. Therefore $K\cap (L+x)$ is bounded,
and $g(x)<\infty$.

\end{proof}

As mentioned in the introduction, the concavity
properties of the cross covariogram
are consequences of the Borell-Brascamp-Lieb inequalities.
For the convenience of the  reader,
the following lemma provides a short proof
of the special case that we need here
(corresponding to $\alpha=+\infty$, $\gamma=1/n$
in~\cite[Cor. 3.5]{BL76}).

\begin{lemma}[$1/n$-concavity] \label{lem:g1n-concave}
Let $g$ be the cross covariogram
of a pair of convex sets $K,L\subset\RR^n$
that satisfy the standing hypothesis.
Then $g^{1/n}$ is concave on its support.
\end{lemma}

\begin{proof}
We need to show that
\begin{equation}
g^{1/n}((1\!-\! t)x_0 + tx_1)
\geq (1-t)g^{1/n}(x_0) + tg^{1/n}(x_1)
\label{eq:g1n-concave}
\end{equation}
for every pair of points $x_0,x_1\in\supp g$, and all $t\in [0,1]$.

For $t\in [0,1]$ set $x_t:= (1-t)x_0+t x_1$
and consider the convex sets
\[
S_t := K \cap (L+x_t).
\]
Since $x_0,x_1\in\supp g$, both $S_0$ and $S_1$ are non-empty.
By construction, $g(x_t)=|S_t|$.
Observe that
\begin{equation}
\label{eq:supset}
S_t \supset (1 - t)S_0 + tS_1\,,
\end{equation}
since the right hand side is contained in
$(1 - t)K + t K = K$ by convexity of $K$, and also
in $(1 - t)(L+x_0) + t(L+x_1) = L+x_t$ by convexity of $L$.
Therefore,
\begin{equation}
\label{eq:supset-volume}
|S_t|\ge |(1 - t)S_0 + tS_1|\,.
\end{equation}
By the Brunn-Minkowski inequality,
\begin{equation}
\label{eq:BM-St}
\left|(1 - t)S_0 + tS_1\right|^{1/n}
\ge  (1-t)|S_0|^{1/n} + t|S_1|^{1/n}
\end{equation}
for all $t\in [0,1]$. We combine
Eqs.~\eqref{eq:supset-volume} and~\eqref{eq:BM-St} and recall
that $g(x_t)=|S_t|$ to arrive at Eq.~\eqref{eq:g1n-concave}.
\end{proof}

\begin{lemma}[Equality conditions]
\label{lem:equality}
In the setting of Lemma~\ref{lem:g1n-concave},
suppose that equality holds in Eq.~\eqref{eq:g1n-concave}
for some $t\in (0,1)$ and some
$x_0, x_1\in\supp g$, $x_0\ne x_1$, such that
at least one among $g(x_0)$ and $g(x_1)$ is positive.
Set $x_t:=(1-t)x_0+tx_1$.
Then $K\cap (L+x_0)$, $K\cap (L+x_1)$,
and $K\cap (L+x_t)$ are homothetic.
\end{lemma}

\begin{proof}
Equality in Eq.~\eqref{eq:g1n-concave}
implies that the inclusion relation~\eqref{eq:supset}
is an identity,
\[
S_t = (1-t)S_0+ t S_1,
\]
and Eq.~\eqref{eq:BM-St} holds with equality.
The claim now follows from the characterization
of the equality cases in the Brunn-Minkowski inequality.
\end{proof}

Since $g^{1/n}$ is concave on its support,
equality in Eq.~\eqref{eq:g1n-concave} for some $t\in (0,1)$
implies that $g$ is affine on the line segment
that connects $x_0$ to $x_1$. By
Lemma~\ref{lem:equality} there exist a convex body $C \subset \RR^n$
such that
$$
K\cap (L+x_t) = \gamma_t C+ v_t,
$$
where $v_t\in\RR^n$ and $\gamma_t\ge 0$
are affine functions of $t$.
The next lemma concerns the special case
of a line segment centered at the origin.

\begin{lemma} [Consequences of Lemma~\ref{lem:equality}]
\label{lem:affine}
Let $g:=g_{K,L}$ be the cross covariogram
of a pair of bodies {that satisfy} the
standing hypothesis,
let
$w$ be a non-zero vector such that $\pm w\in\supp g$, and $\lambda\in [-1,1]$.
If $g(0)>0$ and
\begin{equation}\label{g_affine}
g^{1/n}(\pm w) = (1\pm\lambda) g^{1/n}(0),
\end{equation}
then there exists $v\in\RR^n$ 
such that
\begin{equation}
\label{eq:tv}
(1+t\lambda)(K \cap L) = (K-tv)\cap (L+t(w\!-\!v)),\qquad (-1\le t\le 1).
\end{equation}
\end{lemma}

\begin{proof}
By assumption, $g^{1/n}$ is affine on the line segment
with endpoints $\pm w$; it is positive
except possibly at one endpoint.
Lemma~\ref{lem:equality} implies that
the sets $K\cap (L+tw)$ are homothetic
for $-1< t<1$. This means that
there exists $v\in\RR^n$ 
such that
\[
K\cap (L+tw)= (1+t\lambda)(K\cap L)+tv ,\qquad (-1\le t\le 1),
\]
see Eq.~\eqref{eq:BM-equal}.
Eq.~\eqref{eq:tv} follows upon
translation by $-tv$ and switching sides.
\end{proof}

The implications of Eq.~\eqref{eq:tv}
for the shape of $K$ and $L$
will be discussed in Lemma~\ref{lem:cylinder}
(for $\lambda=0$) and Lemma~\ref{lem:cones}
(for $\lambda>0$) below.
We end this section with a simple observation that
will play a role in the proof of Theorem~\ref{thm:global}.

\begin{lemma}[Intersection of boundaries]\label{lem:dKdL}
Let $K,L\subset \RR^n$ be closed convex
sets such that $K\cap L$ is compact and non-empty.
For $n>1$, if $\partial K\cap \partial L=\emptyset$,
then either $L\subset\inte K$ or $K\subset \inte L$.
\end{lemma}

\begin{proof}
We will use that $\partial (K\cap L)$,
being the boundary of a compact convex set in dimension $n>1$,
is connected.  We write it as the union of two
closed sets,
\[
\partial (K\cap L) = (\partial K\cap L )\cup (K\cap \partial L)\,,
\]
which are disjoint by assumption. By connectedness,
either $\partial K\cap L=\emptyset$ or
$K\cap \partial L=\emptyset$.

In the first case $L$ is a subset of $\RR^n\setminus \partial K$,
which has at most two connected components,
$\inte K$ and the complement of $K$.
Since $L$ is connected, it is entirely contained in one of
these components. But $L\cap K\neq\emptyset$,
and therefore $L\subset \inte K$.

By the same argument,
$K\cap \partial L=\emptyset$ implies that $K\subset\inte L$.
\end{proof}
\section{Strictly convex bodies}
\label{sec:global}

\begin{proof}[Proof of Theorem~\ref{thm:global}]
Let $g:=g_{K,L}$ be the cross covariogram of
two strictly convex bodies $K, L$ in $\RR^n$.
Under the assumption that $\max g<\min\{|K|,|L|\}$,
we want to show that $g^{1/n}$ is strictly
concave on its support.

Suppose, on the contrary,
that $g^{1/n}$ is affine and not identically zero
along some line segment of positive length,
\begin{equation}
\label{eq:affine-1/2}
g^{1/n}(a) =
\frac12\left(g^{1/n}(a-w) + g^{1/n}(a+w)\right).
\end{equation}
Here, $a$ is the midpoint of the line segment, and $w\ne 0$ is
half the distance vector between its endpoints.
By translating $L$ we may take $a=0$.
By strict convexity of $K$ and $L$,
the origin lies in the interior of $K + (-L)$,
and $g(0)>0$.  It follows from
Lemma~\ref{lem:affine} that
Eq.~\eqref{eq:tv} holds for some $v\in\RR^n$ and
$\lambda \in [-1,1]$.
We will show that 
necessarily $w=0$, contradicting the choice of~$w$.

Since $0<g(0) < \min \{\abs{K}, \abs{L}\}$ and $n > 1$,
by Lemma~\ref{lem:dKdL} there exists
a point $z\in \partial K\cap \partial L$.
We express this point as a convex combination in two ways,
\begin{align*}
z&=\ \frac12((1-\lambda)z-v) +
\frac12 ((1+\lambda) z+v)\\
& = \ \frac{1}{2} ((1-\lambda)z-(v\!-\!w)) +
\frac12 ((1+\lambda)z+(v\!-\!w)).
\end{align*}
By Eq.~\eqref{eq:tv}, the points $z \pm (\lambda z+v)$
lie in $K$, while $z\pm (\lambda z +(v\!-\!w))$ lie in~$L$.
Since $z$ is an extreme point of both $K$ and $L$
by the assumption of strict convexity, these
convex combinations are trivial,
that is, $v =w-v=0$.
This proves that strict $1/n$-concavity
cannot fail on a line segment of positive length.
\end{proof}

It is clear from the proof that the assumptions
of Theorem~\ref{thm:global}
can be relaxed in various ways. The conclusions
remain valid for any pair
of strictly convex sets that satisfies the standing hypothesis,
so long as neither set contains a translate of the other
in its interior.  Since $K\cap (L+x)$ is bounded by the
proof of Lemma~\ref{lem:g-support},
it follows from Lemma~\ref{lem:dKdL} that
$\partial K\cap \partial(L+x)\ne \emptyset$
for each $x\in\supp g$.

If one of the sets contains a translate
of the other in its interior,
strict convexity of the sets still implies
{\em local} strict concavity of $g^{1/n}$
in the neighborhood of
any point $x$ with $0<g(x)<\max g$.
Furthermore, all level sets
of $g$ (including its support $K\!-\!L$
and the plateau at the top) are strictly convex.

\section{$g_{K,L}$ constant on a line segment}
\label{sec:cylinder}

\noindent Let $K, L$ be convex sets
that satisfy the standing hypothesis.
In this section, we consider the
situation where their cross covariogram 
equals a positive
constant along a line segment with endpoints $a\pm w$.
This leads to the cylindrical scenario described in the
introduction. As in Section~\ref{sec:global}, we take
$a=0$.

Let $\lambda=0$ and $z\in K\cap L$.  For $v\not\in \{0, w\}$, Lemma~\ref{lem:affine}
guarantees that

\smallskip
\begin{itemize}
        \item
the line segment with endpoints $z\pm v$ lies in $K$, and
\smallskip
        \item the line segment with
endpoints $z\pm (v\!-\!w)$ lies in $L$,
    \end{itemize}
\smallskip
see Eq.~\eqref{eq:tv} with $\lambda=0$.
If $v=0$ or $v=w$, the corresponding segment degenerates
to a point.

The following lemma describes the condition
on $K$ in terms of the cylinder in the direction
of $v$ generated by $K$.
Note that its conclusion holds vacuously also
for $v=0$, where $K'=K$.

\begin{lemma}
\label{lem:cylinder} Let $K,L\subset \RR^n$
be closed convex sets such  that $K\cap L\ne \emptyset$,
let $v\ne 0$ be a vector in  $\RR^n$,
and let $K'=\left\{x+tv\mid x\in K, t\in\RR\right\}$.
If $z\pm v\in K$ for every $z\in K\cap L$, then
\[
K\cap L=K'\cap L\,.
\]
\end{lemma}

\begin{proof}
By definition, $K\subset K'$, and hence
$(K\cap L)\subset (K'\cap L)$.

For the reverse inclusion we argue by contradiction. Consider
$z\in K'\cap L$ and assume that $z\notin K\cap L$.
By definition of $K'$, there exist $x\in K$
and $t\in\RR$, $t\neq0$, such that $z=x+tv$.
We may assume, by switching $v$ with $-v$, that $t>0$.
Let $y$ be the point of $K\cap L$
of minimal distance from $z$,
let $[y,z]$ denote the segment with endpoints
$y$ and $z$, and let $[x,y+v]$ denote the segment
with endpoints $[x,y+v]$.
By convexity, $[x, y+v]\subset K$,
because $x$, $y$, and by the
assumption on $K$ also $y+v$ lie in $K$.
On the other hand,
$[y,z]\subset L$ since both endpoints lie in $K'\cap L$.
These segments are the diagonals
of the trapezoid with vertices
$x$, $z=x+tv$, $y$, and $y+v$.
They intersect at a point of $K\cap L$ that lies
strictly between $z$ and $y$,
$$ \frac{1}{1+t} \,x + \frac{t}{1+t}\, (y+v)
= \frac{1}{1+t} \,z + \frac{t}{1+t}\,y.
$$
This contradicts the choice of $y$
and concludes the proof.
\end{proof}

\begin{proof}[Proof of Theorem~\ref{thm:constant}]
By assumption, the cross covariogram
$g:=g_{K,L}$ is constant and strictly positive on some
line segment with endpoints $a\pm w$. Using a suitable
translation, we take $a=0$.

Lemma~\ref{lem:affine} applies with $\lambda=0$ (because $g$ is constant
on the line segment) and it implies that there exists
$v\in\RR^n$ such that
$K\cap L$ is a subset of both
$K-tv$ and $L+t(w\!-\!v)$ for every $t\in [-1,1]$.
By definition of $K'$ and $L'$ we have the chain of
inclusions
\begin{align}
\label{eq:cylinder-inclusion}
\begin{split}
(K\cap L) & \subset  ((K-sv)\cap (L+t(w\!-\!v)))\\
&\subset (K'\cap L')
\end{split}
\end{align}
for all $s,t\in [-1,1]$.

Since $g(0)>0$, the intersection
$K\cap L$ is non-empty.
We apply Lemma~\ref{lem:cylinder}
twice (once as stated, and the second time with
$L$, $K'$, and $v-w$ in place of $K$, $L$, and $v$)
to see that
\[
K\cap L = K'\cap L= K'\cap L'.
\]
Thus the three intersections in  Eq.~\eqref{eq:cylinder-inclusion}
are equal.  It follows that
\begin{equation}
\label{eq:proof-st}
(K-sv)\cap (L + t(w\!-\!v)) = K'\cap L'\,,\qquad (-1\le s,t\le 1).
\end{equation}
To complete the proof, we set $s=t$ and translate by $tv$.
On the right hand side we use
that $K'+tv=K'$ and $L'+tv = L'+tw$
by the translation symmetries of $K'$ and $L'$.
This proves Eq.~\eqref{eq:KL-prime}
\end{proof}

\begin{remark}
\label{remark:st}
As a consequence of
Eq.~\eqref{eq:proof-st}, $g$ is constant on the convex hull of
the points $a\pm v\pm (w\!-\!v)$.
If $v\not\in \{0,w\}$, this is a parallelogram.
\end{remark}

\section{$g_{K,L}^{1/n}$ non-constant affine on a line segment}
\label{sec:cones}

\noindent Let $K, L$ be convex sets
whose cross covariogram $g:=g_{K,L}$, is affine
(and non-constant) on a line segment.
We relate this situation to the two-cones scenario
described in the introduction.
As in the previous section,
we assume the standing hypothesis
and center the line segment at the origin.

If $g^{1/n}$ is affine
and strictly increasing on the line segment from $-w$ to $w$,
then Lemma~\ref{lem:affine} guarantees that
there exist $v\in\RR^n$ and $\lambda\in (0,1]$
such that, for every $z\in K\cap L$,

\smallskip
\begin{itemize}
        \item
the line segment with endpoints $z\pm(\lambda z+ v)$ lies in $K$, and
\smallskip
        \item the line segment with
endpoints $z\pm (\lambda z + (v\!-\!w))$ lies in $L$,
    \end{itemize}
\smallskip
see Eq.~\eqref{eq:tv}.
The segments lie on lines passing
through $-\frac1\lambda v$ and  $\frac1\lambda(w\!-\!v)$,
respectively.  The next lemma describes these families of
lines in terms of the convex cones $K''$ and $L''$.

\begin{lemma}
\label{lem:cones}
Let $K,L\subset \RR^n$
be closed convex sets such that $K\cap L\ne \emptyset$.
Let $v\in \RR^n$,
let $\lambda\in (0,1]$, and let
$K''=\left\{-\tfrac1\lambda v +t(x+\tfrac1\lambda v)
\,\big\vert\, x\in K, t\ge 0 \right\}$.
If $z\pm (\lambda z +v)\in K$ for
every $z\in K\cap L$, then
\[  K\cap L=K''\cap L.
\]
\end{lemma}

\begin{proof}We proceed
as for Lemma~\ref{lem:cylinder}.
Clearly $(K\cap L)\subset (K''\cap L)$.

For the reverse inclusion, define $x, y$ and $z$
in analogy to the proof of Lemma~\ref{lem:cylinder}. We observe that
$y\neq -\frac1{\lambda}v$, because otherwise,
by definition of $K''$,  $x$ belongs to $[y, z]\setminus\{y\}$,
which contradicts the minimality property of $y$.
After this clarification, the proof proceeds as in Lemma~\ref{lem:cylinder}.
\end{proof}

\begin{proof}[Proof of Theorem~\ref{thm:affine}]
By assumption, the cross covariogram
$g:=g_{K,L}$ is affine and
non-constant on some line segment
with endpoints $a\pm w$, where $w\ne 0$.
We translate $L$ so that $a=0$; switching the endpoints, if
necessary, we take $g(w)>g(-w)$, so that
Eq.~\eqref{g_affine} holds for some $\lambda\in(0,1]$.
Since $g(0)>0$, the intersection
$K\cap L$ is non-empty.

Lemma~\ref{lem:affine} implies that there exists $v\in\RR^n$  such that
\begin{align}
\label{eq:affine-proof}
\begin{split}
(1+t\lambda) (K\cap L)
&= ((K-tv)\cap (L+t(w\!-\!v)))\\
& \subset (K''-tv)\cap (L''+t(w\!-\!v))
\end{split}
\end{align}
for all $t\in [-1,1]$.
Here $L''$ is the cone with vertex $-\frac1{\lambda}(v-w)$
generated by $L$.

By the symmetry of the cones under dilation from their vertices,
\[
((K''-tv)\cap (L''+t(w\!-\!v)))\ \subset \ (1+t\lambda)(K''\cap L''),
\qquad (-1\le t\le 1),
\]
with equality for $t\lambda\ne -1$.
On the other hand, we can use Lemma~\ref{lem:cones} twice
to see that
\[
K\cap L = K''\cap L= K''\cap L''.
\]
Therefore the inclusion relation in
Eq.~\eqref{eq:affine-proof} is
an equality.  Eq.~\eqref{eq:KL-pprime} follows
upon translation by $tv$.
\end{proof}

\section{Examples}
\label{sec:examples}

For strictly convex origin-symmetric bodies $K, L\subset\RR^n$,
Scenarios (2) and (3) cannot occur
and we have the following 
characterization of strict $1/n$-concavity:

\begin{corr} [of Theorem~\ref{thm:global}]
\label{corr:symmetric}
Let $K, L\subset\RR^n$ be origin-symmetric,
strictly convex bodies of dimension $n>1$,
and $g:=g_{K,L}$ their cross covariogram.
The following are equivalent:
\begin{enumerate}
\item[(a)]  $g^{1/n}$ is is strictly concave on its support;
\item[(b)] $g$ attains its maximum only at the origin;
\item[(c)]  $\partial K\cap \partial L\ne \emptyset$.
\end{enumerate}
\end{corr}

\begin{proof}
$(a)\Rightarrow(b)$:
By strict concavity $g^{1/n}$ attains its maximum at a unique point.
By symmetry, this point is the origin.

\smallskip\noindent
$(b)\Rightarrow(c)$:
If $\partial K\cap \partial L=\emptyset$,
then by Lemma~\ref{lem:dKdL} one of the bodies (say, $K$)
contains the other ($L$) in its interior.
This is Scenario~(1), where $g$ has a plateau at height
$|L|$.

\smallskip\noindent
{
$(c)\Rightarrow (a)$:
Let $a\in \supp g$. By Lemma~\ref{lem:g-support},
$K\cap (L+a)\ne\emptyset$. We claim that their boundaries
intersect as well.
}

{
Suppose, for the contrary, that
$\partial K \cap \partial (L+a)=\emptyset$.
By Lemma~\ref{lem:dKdL},
one of the bodies $K$ and $L+a$ contains the other
in its interior.  If $L+a\subset \inte K$, then
also $L-a = -(L-a) \subset \inte(-K) = \inte K$ by the symmetry of
$K$ and $L$, and hence
\[
L = \frac12\left((L-a) + (L+a)\right) \subset \inte K\,
\]
by convexity of $K$ and $L$.
If, instead, $K\subset \inte (L+a)$ then
$K-a \subset\inte L$, and
\[
K = \frac12\left((K-a) + (K+a)\right) \subset \inte L\,.
\]
In both cases $\partial K\cap \partial L=\emptyset$,
contradicting (c). Thus $\partial K \cap \partial (L+a)\ne \emptyset$.}

Choose now $z\in\partial K \cap\partial (L+a)$, and
proceed as in the proof of Theorem~\ref{thm:global}
to see that $g^{1/n}$ is not
affine on any line segment of positive length centered at~$a$.
\end{proof}

We next remove the symmetry assumption
from the theorem of Meyer, Reisner, and
Schmuckenschl\"ager 
that was stated after Theorem~\ref{thm:global}.

\begin{corr} [of Theorem~\ref{thm:constant}]
\label{corr:minus}
Let $K\subset \RR^n$ be a closed convex set,
$-K$ its reflection
through the origin, and $g:=g_{K,-K}$
their cross covariogram.
If the pair $(K, -K)$ satisfies the standing hypothesis,
then the level sets
\[
\{x\in\RR^n\mid g(x)\ge h\},\qquad (0<h<\max g)
\]
are strictly convex. Moreover, $g$ attains its maximum
at a unique point.
\end{corr}

\begin{proof}[Proof]
Suppose that for some $h\in (0,\max g]$,
the set $\{x\in\RR^n\mid g(x)= h\}$
contains a line segment with endpoints $a\pm w$,
where $w\ne 0$. Translating $K$ by
$-\frac12 a$, we take $a=0$.

Since $g(x)\equiv h>0$ along the line segment,
we can apply Theorem~\ref{thm:constant}
and Eq.~\eqref{eq:KL-prime-b}
with $L=-K$
to see that there exists a vector $v$ such that
\[
K\cap (-K +tw) = K'\cap ((-K)'+tw) = K \cap (-K)+tv
\]
for all $t\in [-1,1]$.
Here, $K'$ and
$(-K)'$ are the cylinders defined
in Theorem~\ref{thm:constant}, with $L=-K$.
After subtracting $\frac{t}2 w$ on both sides,
we observe that the left hand side has become symmetric
through the origin, while the right hand side has become symmetric through $t(v-\frac1{2}w)$. Since the sets are bounded and convex
they can have at most one center of symmetry,
and we infer that $v=\frac12 w$.

It follows that every supporting hyperplane for
$K$ at a point in $\partial K\cap (-K)$ contains
a segment in the direction of $w$.
By symmetry, the same is true for
points in $K\cap \partial(-K)$.
In summary, every point in $\partial (K\cap (-K))$
has a supporting hyperplane with normal in $w^\perp$.
Since the range of the Gauss map
of $K\cap (-K)$ is contained in a hyperplane,
we conclude that $\inte(K\cap (-K))=\emptyset$,
contrary to the assumption that $g>0$
on the line segment.
\end{proof}

Since the conclusion of Corollary~\ref{corr:minus}
rules out Scenarios (1-2), it implies
that $g_{K,-K}$ is strictly
$\alpha$-concave for all $\alpha<1/n$ on the interior of its
support, {$2K$.}
However, Scenario (3) commonly occurs for $g_{K,-K}$,
e.g. when $K$ is a polytope.
Note that
$g_{K,-K}$ is not symmetric through the origin
unless $K$ is itself symmetric;
moreover, $\max g_{K,-K}<|K|$ unless
$-K$ is a translate of $K$.

\begin{figure}[t]
    \centering
    \includegraphics[height=4cm]{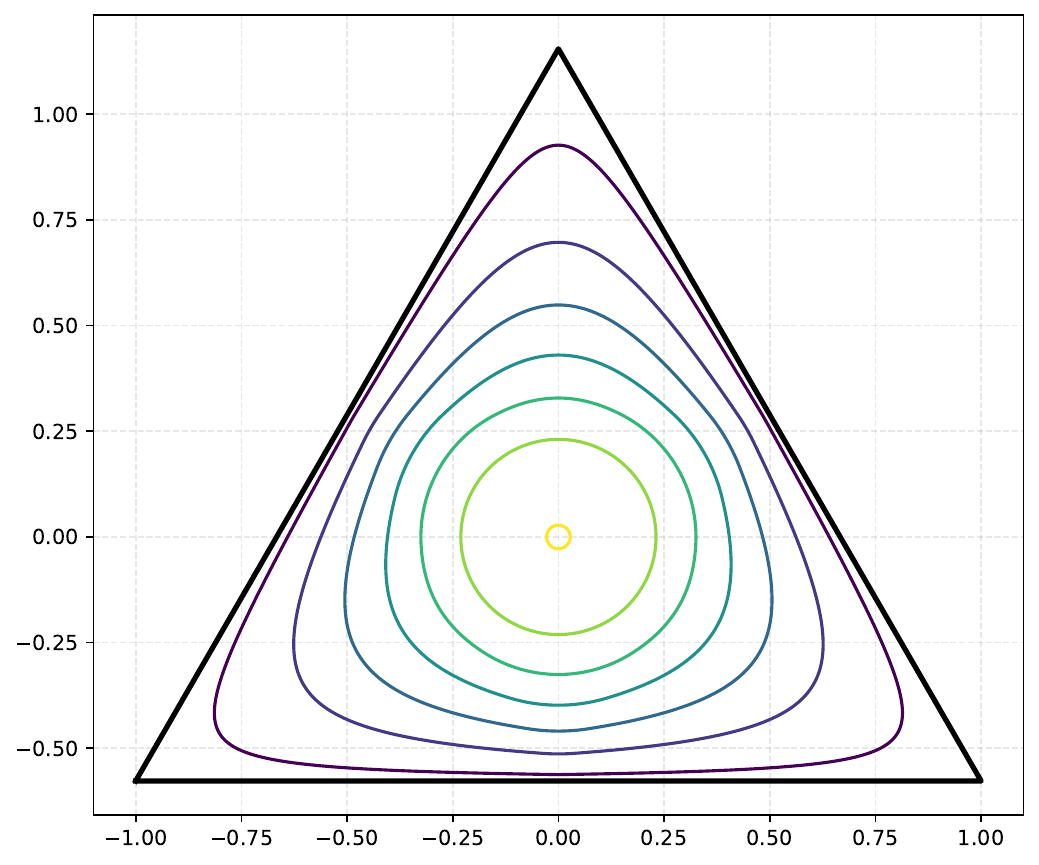}
\qquad \qquad
    \includegraphics[height=4cm]{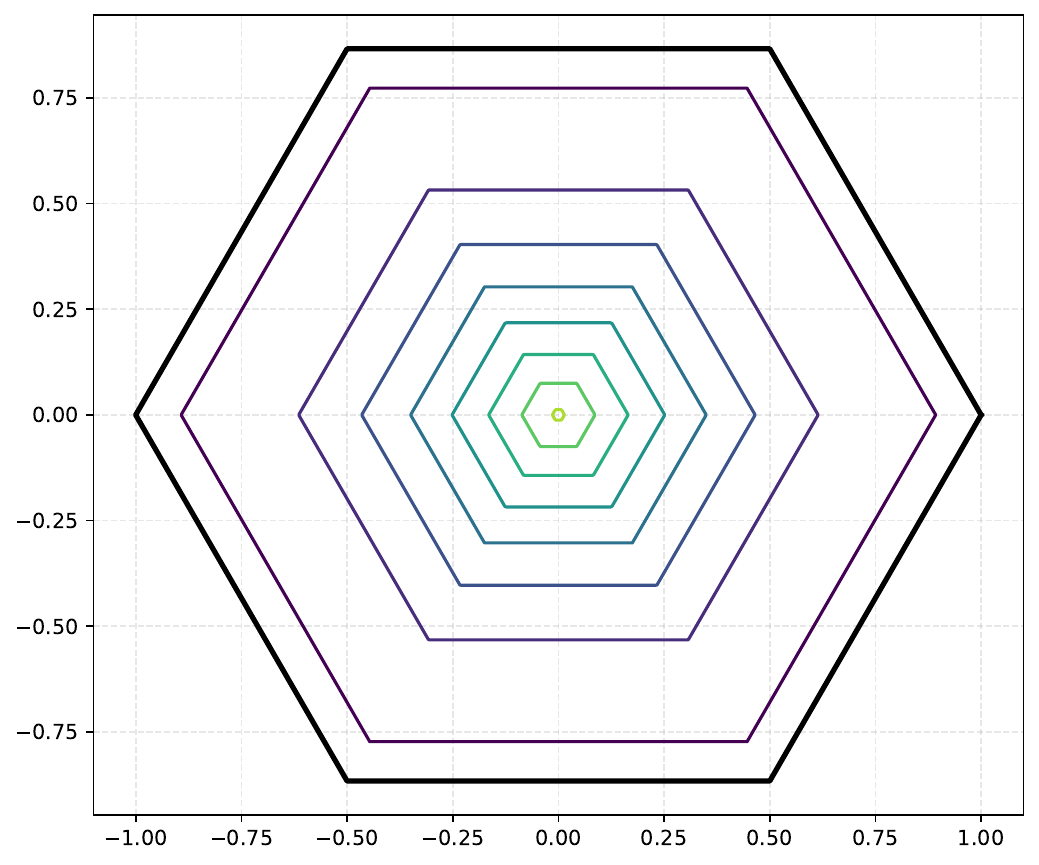}
    \caption{\small
Contour plots of the
cross covariogram $g_{T,-T}$ ({\em left})
and
the covariogram $g_{T,T}$ ({\em right})
for an equilateral triangle
of side length~1, centered at the origin}
\label{fig:triangle}
\end{figure}

We illustrate these conclusions
with the example of an equilateral triangle, $T$.
Fig.~\ref{fig:triangle} shows the
cross covariogram $g_{T,-T}$
(on the left)
in comparison with the covariogram
$g_{T,T}$ (right).
By Corollary~\ref{corr:minus} the cross covariogram $g_{T,-T}$
has strictly convex level sets.
According to Scenario (3), $g_{T,-T}^{1/2}$
increases linearly along rays that connect
a vertex of $T+T$ half-way to the opposite side.
In contrast, $g_{T,T}^{1/2}$ is a piecewise
affine function with the symmetries of a regular hexagon;
Scenarios (2) and (3) occur
simultaneously on a dense open subset of its support.
%
Although both functions integrate to $|T|^2$,
the cross covariogram $g_{T,-T}$
has smaller support and a lower maximum than $g_{T,T}$.

As a last example, we consider
the covariogram of a convex set $K$ with a
{strictly} convex set $L$.
Since $\partial L$ contains no line segments,
Scenario (3) cannot occur.
The only obstructions to strict $1/n$-concavity are
Scenario (1), and Scenario (2) with $v=w$
(see Figs.~\ref{fig:containment} and~\ref{fig:cylinder}, left).

\begin{corr}[of Theorems~\ref{thm:constant} and~\ref{thm:affine}]
Let $K,L$ be convex sets, and $g:=g_{K,L}$
their cross covariogram.
Assume in addition to the standing hypothesis that
$L$ is strictly convex.
If $g^{1/n}$ is affine and not identically zero
on some line segment with endpoints
$a\pm w$, then it is
constant on that line segment, and
\[
K\cap (L+a+tw)= K'\cap (L+a+tw) = (K\cap (L+a)) +tw
\]
for all $-1\le t\le 1$.  Here $K'$
is the cylinder in the direction of~$w$ generated by $K$.
In particular,
for each $z\in (\partial K)\cap (L+a)$
the boundary of $K$ contains the line segment with endpoints
$z\pm w$.
\end{corr}

\bibliographystyle{amsplain}
\bibliography{AGL}

\end{document}